\documentclass[12pt]{amsart}
\usepackage{amscd}      
\usepackage{amssymb, amsthm}
\usepackage{amsmath, graphics}
\usepackage{xypic}      
\LaTeXdiagrams          
\usepackage[all,v2]{xy}
\xyoption{2cell} \UseAllTwocells \xyoption{frame} \CompileMatrices
\allowdisplaybreaks[3]

\usepackage{latexsym}
\usepackage{epsfig}
\usepackage{amsfonts}
\usepackage{enumerate}
\usepackage{times}

\newtheorem{prop}{Proposition}[section]
\newtheorem{lem}[prop]{Lemma}

\newtheorem{cor}[prop]{Corollary}
\newtheorem{thm}[prop]{Theorem}

\newtheorem{defn}[prop]{Definition}

\newtheorem{rmk}[prop]{Remark}

            {\nolinebreak $\Box$ \end{trivlist}}

\newcommand{\noprint}[1]{}

\renewcommand{\tilde}{\widetilde}

\newcommand{\toto}{\rightrightarrows}

\newcommand{\ldiag}[1]%
       {\makebox[0cm]{${\scriptstyle#1}\downarrow\phantom{\scriptstyle#1}$}}
\newcommand{\ldiagup}[1]%
       {\makebox[0cm]{${\scriptstyle#1}\uparrow\phantom{\scriptstyle#1}$}}
\newcommand{\rdiag}[1]%
       {\makebox[0cm]{$\phantom{\scriptstyle#1}\downarrow{\scriptstyle#1}$}}
\newcommand{\sediagr}[1]%
       {\makebox[0cm]{$\phantom{\scriptstyle#1}\searrow{\scriptstyle#1}$}}
\newcommand{\nediagr}[1]%
       {\makebox[0cm]{$\phantom{\scriptstyle#1}\nearrow{\scriptstyle#1}$}}
\newcommand{\rdiagup}[1]%
       {\makebox[0cm]{$\phantom{\scriptstyle#1}\uparrow{\scriptstyle#1}$}}
\newcommand{\swdiag}[1]%
       {\makebox[0cm]{$\phantom{\scriptstyle#1}\swarrow{\scriptstyle#1}$}}
\newcommand{\sediag}[1]%
       {\makebox[0cm]{${\scriptstyle#1}\searrow\phantom{\scriptstyle#1}$}}
\newcommand{\nediag}[1]%
       {\makebox[0cm]{${\scriptstyle#1}\nearrow\phantom{\scriptstyle#1}$}}

\newcommand{\doublearrowstack}[2]%
                      {{{{\scriptstyle#1}\atop{\textstyle\longrightarrow}}\atop{{\textstyle\longrightarrow}\atop{\scriptstyle#2}}}}
\newcommand{\rightleftarrowstack}[2]%
                      {{{{\scriptstyle#1}\atop{\textstyle\longrightarrow}}\atop{{\textstyle\longleftarrow}\atop{\scriptstyle#2}}}}
\newcommand{\leftrightarrowstack}[2]%
                      {{{{\scriptstyle#1}\atop{\textstyle\longleftarrow}}\atop{{\textstyle\longrightarrow}\atop{\scriptstyle#2}}}}

\newcommand{\overtoparrow}%
{\makebox[0cm]{\beginpicture \setcoordinatesystem units
<.8cm,.4cm> point at 0 0 \setplotarea x from -3 to 3, y from 0 to
1 \setquadratic \plot -3 0 0 1 3 0 / \put{\vector(3,-1){0}}[Bl] at
3 0
\endpicture}}

\newcommand{\underbottomarrow}%
{\makebox[0cm]{\beginpicture \setcoordinatesystem units
<.8cm,.4cm> point at 0 0 \setplotarea x from -3 to 3, y from 0 to
1 \setquadratic \plot -3 1 0 0 3 1 / \put{\vector(3,1){0}}[Bl] at
3 1
\endpicture}}

\newcommand{\ses}[5]%
{0\longrightarrow#1\stackrel{#2}{ \longrightarrow}#3\stackrel{#4}{
\longrightarrow}#5\longrightarrow0}

\newcommand{\dt}[6]%
{#1\stackrel{#2}{longrightarrow}#3
\stackrel{#4}{\longrightarrow}#5 \stackrel{#6}{\longrightarrow}
#1[1]}

\newcommand{\cat}[1]%
{(\mbox{\rm #1})}

\def\Label#1{\label{#1}{\tt [#1]}\phantom{h}}
\def\Label{\label}

\begin{document}

\title{On the $K$-theory of Toric Stack Bundles}

\author{Yunfeng Jiang}
\address{Department of Mathematics\\ University of Utah\\ 155 S 1400 E JWB 233\\Salt Lake city\\ UT 84112\\ USA}
\email{jiangyf@math.utah.edu}

\author{Hsian-Hua Tseng}
\address{Department of Mathematics\\ University of Wisconsin-Madison\\ Van Vleck Hall, 480 Lincoln Drive \\Madison\\ WI 53706-1388 \\ USA}
\email{tseng@math.wisc.edu}

\date{\today}

\begin{abstract}
Simplicial toric stack bundles are smooth Deligne-Mumford stacks over smooth varieties with fibre a toric Deligne-Mumford stack.  We compute the Grothendieck $K$-theory of simplicial toric stack bundles and study the Chern character homomorphism.
\end{abstract}
\maketitle

\section{Introduction}
Simplicial toric stack bundles, as defined in \cite{Jiang}, are bundles over a smooth base variety $B$
with fibers toric Deligne-Mumford stacks in the sense of \cite{BCS}. In this paper we compute the Grothendieck $K$-theory of simplicial toric stack bundles. 

In \cite{Jiang}, the construction of toric Deligne-Mumford stacks was slightly generalized by extending the notion of 
stacky fans. A stacky fan\footnote{In \cite{Jiang} this is called an {\em extended} stacky fan.} is a triple
$\mathbf{\Sigma}:=(N,\Sigma,\beta)$, where $N$ is a finitely generated abelian group\footnote{We denote by $N_{tor}$ the torsion subgroup of $N$.}
 of rank $d$,  $\Sigma$ is a simplicial fan in the lattice $\overline{N}=N/N_{tor}\subset N_\mathbb{Q}$, and 
$\beta:
\mathbb{Z}^{m}\to N$ is a map determined by integral vectors
$b_{1},\ldots,b_{n},b_{n+1},\ldots,b_{m}\in N$ ($m\geq n$) satisfying the
condition that for $1\leq i\leq n$ the image $\overline{b}_{i}\in \overline{N}$ under the projection
$N\to \overline{N}$ generates the ray $\rho_{i}\in\Sigma$.  We call
$\{b_{n+1},\cdots,b_{m}\}$ the extra data in
$\mathbf{\Sigma}$. The stacky fan $\mathbf{\Sigma}$ yields an exact sequence,
\begin{equation}\label{exact}
1\longrightarrow \mu\longrightarrow
G\stackrel{\alpha}{\longrightarrow}
(\mathbb{C}^{*})^{m}\longrightarrow T\longrightarrow 1
\end{equation}
where $T=(\mathbb{C}^{*})^{d}$. We associated to $\mathbf{\Sigma}$ a {\em toric Deligne-Mumford stack}
$\mathcal{X}(\mathbf{\Sigma}):=[Z/G]$, where $Z=(\mathbb{C}^{n}\setminus\mathbb{V}(J_{\Sigma}))\times (\mathbb{C}^{*})^{m-n}$, the 
ideal $J_{\Sigma}$ is the irrelevant ideal of the fan $\Sigma$,  and $G$ acts on $Z$ via the homomorphism $\alpha: G\to (\mathbb{C}^{*})^{m}$ above.

Removing the extra data $\{b_{n+1},\cdots,b_{m}\}$ from the map $\beta$ yields $\beta_{min}: \mathbb{Z}^{n}\to N$ given by the integral vectors
$\{b_1,\cdots,b_{n}\}$. The triple $\mathbf{\Sigma_{min}}:=(N,\Sigma,\beta_{min})$ is the stacky fan 
in the sense of \cite{BCS}. The toric Deligne-Mumford stack $\mathcal{X}(\mathbf{\Sigma_{min}})$ is isomorphic to $\mathcal{X}(\mathbf{\Sigma})$, see \cite{Jiang}. The stacky fan $\mathbf{\Sigma_{min}}$ may be interpreted as the minimal representation of the associated toric Deligne-Mumford stack.

Let $P\to B$ be a principal $(\mathbb{C}^{*})^{m}$-bundle,  let
$^{P}\mathcal{X}(\mathbf{\Sigma})$  be the quotient stack
$[(P\times_{(\mathbb{C}^{*})^{m}}Z)/G]$, where
$G$ acts on $B$ trivially and on $(\mathbb{C}^{*})^{m}$
via the map $\alpha$ above.  Then
$^{P}\mathcal{X}(\mathbf{\Sigma})$ is a toric stack bundle
over $B$ with fibre the  toric Deligne-Mumford stack
$\mathcal{X}(\mathbf{\Sigma})$.  The extra data
$\{b_{n+1},\cdots,b_{m}\}$ in $\mathbf{\Sigma}$ can be put
into the $Box(\mathbf{\Sigma})$ which do not influence the
structure of the toric stack bundle
$^{P}\mathcal{X}(\mathbf{\Sigma})$. The choice of torsion and
nontorsion extra data does affect the structure of
$^{P}\mathcal{X}(\mathbf{\Sigma})$, but not the Chen-Ruan (orbifold) cohomology,
see \cite{Jiang}.

Let $\rho_{i}\in \Sigma$ be a ray. There is a corresponding line bundle $\mathcal{L}_{i}$ over $^{P}\mathcal{X}(\mathbf{\Sigma})$,  which is  the trivial line bundle $\mathbb{C}$ over $P\times_{(\mathbb{C}^{*})^{m}}Z$ with the $G$-action given by the $i$-th component of the map $\alpha$. The ray $\rho_i$ also defines a line bundle $L_i$ over $\mathcal{X}(\mathbf{\Sigma})$ via the $i$-th component of $\alpha$. The line bundle $\mathcal{L}_{i}$ can be taken as the twist $P(L_i)$ of $L_i$  by the principle $(\mathbb{C}^{*})^{m}$-bundle $P$.

Let $R$ denote the character ring of the group $G_{min}$, which is isomorphic to $DG(\beta_{min})$ in the Gale dual map $\beta_{min}^{\vee}: \mathbb{Z}^{n}\to  DG(\beta_{min})$.  Every character $\chi\in R$  gives a line bundle $\mathcal{L}_{\chi}$ over $^{P}\mathcal{X}(\mathbf{\Sigma})$. The line bundle
$\mathcal{L}_{i}$ is given by the standard character $\chi_{i}$ induced by the standard generator $x_i$ on $\mathbb{Z}^{n}$. We let $x_{i}$ represent the class $[\mathcal{L}_{i}]$ in the $K$-theory. 
Let $M=N^{\star}$ be the dual of $N$. For $\theta\in
M$, let $\xi_{\theta}\to B$ be the line bundle coming
from the principal $T$ bundle $E\to B$ by ``extending''
the structure group via $\chi^{\theta}:T\to
\mathbb{C}^{*}$, where $E\to B$ is induced from
the $(\mathbb{C}^{*})^{m}$-bundle $P$ via the map $(\mathbb{C}^{*})^{m}\to T$ in (\ref{exact}).  Let 
$\{v_1,\cdots,v_d\}$ be a basis of $\overline{N}=\mathbb{Z}^{d}$, we choose a basis $\{u_1,\cdots,u_d\}$ of $M$,  which is dual to $\{v_1,\cdots,v_d\}$.
Write $\xi_{i}=\xi_{u_{i}}$.

Let $K(B)$ be the $K$-theory ring of the smooth variety $B$. Let $C(^{P}\mathbf{\Sigma})$ be the ideal in the ring $K(B)\otimes R$
generated by the elements
\begin{equation}\label{ideal}
\left(\prod_{1\leq j\leq n}x_{j}^{\langle\theta,b_{j}\rangle}-\prod_{1\leq i\leq d}(\xi^{\vee}_{i})^{\langle\theta,v_{i}\rangle}\right)_{\theta\in
M},
\end{equation} 
where $\xi^{\vee}_{i}$ is the dual of the line bundle $\xi_i$.
Let $I_{\mathbf{\Sigma}}$ be the ideal generated by
\begin{equation}\label{ideal2}
\prod_{i\in I}(1-x_i)
\end{equation}
where $I\subseteq[1,\cdots,n]$ such that $\{\rho_i| i\in I\}$ do not form a cone 
in $\Sigma$.

\begin{thm}\label{main}
Let $K_{0}(^{P}\mathcal{X}(\mathbf{\Sigma}))$ be the Grothendieck $K$-theory ring of the toric stack bundle $^{P}\mathcal{X}(\mathbf{\Sigma})$. Then the morphism
$$\phi:  \frac{K(B)\otimes R}{I_{\mathbf{\Sigma}}+C(^{P}\mathbf{\Sigma})}\longrightarrow K_{0}(^{P}\mathcal{X}(\mathbf{\Sigma})),$$
which send $\chi$ to $[\mathcal{L}_{\chi}]$, is an isomorphism.
\end{thm}

In the reduced case, i.e. the abelian group $N$ is torison-free, the toric Deligne-Mumford stack $\mathcal{X}(\mathbf{\Sigma})$ is an orbifold. Then every character of $G$ can be lifted to a 
character of $(\mathbb{C}^{*})^{n}$. We have the corollary:
 
 \begin{cor}\label{main2}
Let $K_{0}(^{P}\mathcal{X}(\mathbf{\Sigma}))$ be the Grothendieck $K$-theory ring of the toric stack bundle $^{P}\mathcal{X}(\mathbf{\Sigma})$ with  $\mathcal{X}(\mathbf{\Sigma})$ a reduced toric Deligne-Mumford stack. Then the morphism
$$\phi:  \frac{K(B)[x_1,x_{1}^{-1},\cdots,x_{n},x_{n}^{-1}]}{I_{\mathbf{\Sigma}}+C(^{P}\mathbf{\Sigma})}\longrightarrow K_{0}(^{P}\mathcal{X}(\mathbf{\Sigma})),$$
which send $x_{i}$ to $[\mathcal{L}_{i}]$, is an isomorphism.
\end{cor}

Our proof of the main theorem is based on computations of the $K$-theory rings of toric Deligne-Mumford stacks \cite{BH}, and of toric bundles \cite{SU}. 

This paper is organized as follows. The basic construction of toric stack
bundles defined in \cite{Jiang} is reviewed in Section \ref{review}. Chen-Ruan orbifold cohomology ring of toric stack bundles is discussed in Section \ref{CR-ring}. In Section \ref{k-theory} we compute the $K$-theory ring of toric stack bundles, and in Section \ref{Chern-Character} we show that there is a Chern character isomorphism from the $K$-theory of the toric stack bundle to the Chen-Ruan cohomology ring. In Section \ref{gerbes} we give an interesting example, where we compute the $K$-theory ring of finite abelian gerbes over smooth varieties and compare with the Chen-Ruan cohomology calculated in \cite{Jiang}.

\subsection*{Conventions}
In this paper we work algebraically over the field of
complex numbers.  We use the rational numbers $\mathbb{Q}$ as
coefficients of (orbifold) Chow ring and (orbifold) cohomology ring.  
By an orbifold we mean a smooth
Deligne-Mumford stack with trivial generic stabilizer.
We refer to \cite{BCS} for the construction of Gale dual
$(\beta)^{\vee}: \mathbb{Z}^{m}\to DG(\beta)$ from 
$\beta: \mathbb{Z}^{m}\to N$. We write 
$\mathbb{C}^{*}=\mathbb{C}\setminus \{0\}$. $N^{\star}$ denotes the dual of 
$N$ and $N\to \overline{N}$ is the natural map
modulo torsion. 

For the cones in $\Sigma$, we assume that the rays $\rho_1,\cdots,\rho_d$ span a top dimensional 
cone $\sigma\in\Sigma$, and $\rho_{d+1},\cdots,\rho_{n}$ are the other rays.  Let $v_i\in\rho_i$ be such that $\{v_1,\cdots,v_d\}$ is a basis of $\overline{N}=\mathbb{Z}^{d}$. Let $\{u_1,\cdots,u_d\}$ be the dual basis in $M=N^{\star}$.

\subsection*{Acknowledgments}
Y. J. thanks the Institute of Mathematics in Chinese Academy of Science for financial support during a visit in May, 2008, where part of this work was done. H.-H. T. is supported in part by NSF grant DMS-0757722.

\section{Toric Stack Bundles}\Label{review}
In this section we review the basic construction of toric stack bundles, see \cite{Jiang} for details.
\subsection{Toric Deligne-Mumford Stacks}
  Let $N$ be a finitely generated abelian group of rank $d$ and  $\overline{N}=N/N_{tor}$  the lattice generated by
$N$ in the $d$-dimensional vector space
$N_{\mathbb{Q}}:=N\otimes_{\mathbb{Z}}\mathbb{Q}$. Write $\overline{b}$ for 
the image of  $b$ under the natural map
$N\to \overline{N}$.  Let $\Sigma$ be a rational simplicial fan in
$N_{\mathbb{Q}}$. Suppose  $\rho_{1},\ldots,\rho_{n}$ are the rays
in $\Sigma$. We fix $b_{i}\in N$ for $1\leq i\leq n$ such that
$\overline{b}_{i}$ generates the ray $\rho_{i}$. Let
$\{b_{n+1},\ldots,b_{m}\}\subset N$. We  consider the
homomorphism $\beta: \mathbb{Z}^{m}\to N$
determined by the elements $\{b_{1},\ldots,b_{m}\}$. We require
that $\beta$ has finite cokernel.
\begin{defn}
The triple $\mathbf{\Sigma}:=(N,\Sigma,\beta)$ is called a
stacky fan.
\end{defn}
\begin{rmk}
If $m=n$, then $\mathbf{\Sigma}$ is the stacky fan in the sense of Borisov-Chen-Smith \cite{BCS}.
\end{rmk}

The stacky fan $\mathbf{\Sigma}$ determines two exact sequences:
$$0\longrightarrow DG(\beta)^{\star}\longrightarrow \mathbb{Z}^{m}\stackrel{\beta}{\longrightarrow} N\longrightarrow Coker(\beta)\longrightarrow
0,$$
$$0\longrightarrow N^{\star}\longrightarrow \mathbb{Z}^{m}\stackrel{\beta^{\vee}}{\longrightarrow}
DG(\beta)\longrightarrow
Coker(\beta^{\vee})\longrightarrow 0,$$ where
$\beta^{\vee}$ is the Gale dual of $\beta$. As a
$\mathbb{Z}$-module, $\mathbb{C}^{*}$ is divisible, so it is
an injective $\mathbb{Z}$-module, and hence the functor
$\text{Hom}_{\mathbb{Z}}(-,\mathbb{C}^{*})$ is exact (see e.g \cite{Lang}).  This yields an exact sequence:
$$1\to \text{Hom}_{\mathbb{Z}}(Coker(\beta^{\vee}),\mathbb{C}^{*})\to
\text{Hom}_{\mathbb{Z}}(DG(\beta),\mathbb{C}^{*})\to
\text{Hom}_{\mathbb{Z}}(\mathbb{Z}^{m},\mathbb{C}^{*})\to
\text{Hom}_{\mathbb{Z}}(N^{\star},\mathbb{C}^{*})\to 1.$$
Wrtie
$\mu:=\text{Hom}_{\mathbb{Z}}(Coker(\beta^{\vee}),\mathbb{C}^{*})$, $G:=\text{Hom}_{\mathbb{Z}}(DG(\beta),\mathbb{C}^{*})$, $T:=\text{Hom}_{\mathbb{Z}}(N^{\star},\mathbb{C}^{*})$, then the above sequence reads 
\begin{equation}\label{stack}
1\longrightarrow \mu\longrightarrow
G\stackrel{\alpha}{\longrightarrow}
(\mathbb{C}^{*})^{m}\longrightarrow T\longrightarrow 1,
\end{equation}
which is the same as (\ref{exact}). Define $Z=(\mathbb{C}^{n}\setminus \mathbb{V}(J_{\Sigma}))\times (\mathbb{C}^{*})^{m-n}$,
where $J_{\Sigma}$ is the irrelevant ideal of the fan $\Sigma$.
There exists a natural action of $(\mathbb{C}^{*})^{m}$
on $Z$. The group $G$ acts on $Z$ through the map
$\alpha$ in (\ref{stack}). The quotient stack $[Z/G]$
is associated to the groupoid $Z\times G\toto Z$.
The morphism $\varphi: Z\times G\to
Z\times Z$ to be $\varphi(x,g)=(x,g\cdot x)$ is
finite, hence $[Z/G]$ is a Deligne-Mumford stack. 

\begin{defn}
For a stacky fan $\mathbf{\Sigma}=(N,\Sigma,\beta)$, define $\mathcal{X}(\mathbf{\Sigma}):=[Z/G]$.
\end{defn}

Let $\mathbf{\Sigma}$ be a stacky fan. Let $\beta_{min}: \mathbb{Z}^{n}\to N$ be the map given by the first $n$ integral vectors $\{b_1,\cdots,b_n\}$
in the map $\beta$. Then $\mathbf{\Sigma_{min}}=(N,\Sigma,\beta_{min})$
is a stacky fan, which we call the minimal stacky fan.
From the definitions, we have the following
commutative diagram:
\[
\begin{CD}
0 @ >>>\mathbb{Z}^{n}@ >>> \mathbb{Z}^{m}@ >>> \mathbb{Z}^{m-n} @
>>> 0\\
&& @VV{\beta_{min}}V@VV{\beta}V@VV{\widetilde{\beta}}V \\
0@ >>> N @ >{id}>>N@ >>> 0 @>>> 0.
\end{CD}
\]
From the definition of Gale dual, we compute that
$DG(\widetilde{\beta})=\mathbb{Z}^{m-n}$ and
$\widetilde{\beta}^{\vee}$ is an isomorphism. So by Lemma 2.3
in \cite{BCS}, applying the Gale dual yields
\begin{equation}
\begin{CD}
0 @ >>>\mathbb{Z}^{m-n}@ >>> \mathbb{Z}^{m}@ >>> \mathbb{Z}^{n} @
>>> 0\\
&& @VV{\widetilde{\beta}^{\vee}}V@VV{\beta^{\vee}}V@VV{\beta_{min}^{\vee}}V \\
0@ >>> \mathbb{Z}^{m-n} @ >{}>>DG(\beta)@ >>> DG(\beta_{min}) @>>> 0.
\end{CD}
\end{equation}
Taking
$\text{Hom}_{\mathbb{Z}}(-,\mathbb{C}^{*})$ functor, we get
\begin{equation}\Label{second}
\begin{CD}
1 @ >>>G_{min}@ >{\varphi_{1}}>> G@ >>> (\mathbb{C}^{*})^{m-n}
@
>>> 1\\
&& @VV{\alpha_{min}}V@VV{\alpha}V@VV{\widetilde{\alpha}}V \\
1@ >>>(\mathbb{C}^{*})^{n} @ >>>(\mathbb{C}^{*})^{m}@
>>>(\mathbb{C}^{*})^{m-n} @>>> 1.
\end{CD}
\end{equation}

Let $\varphi_{0}: \mathbb{C}^{n}\setminus \mathbb{V}(J_{\Sigma})\to Z$ be the inclusion
defined by $z\mapsto (z,1)$. So $$(\varphi_{0}\times
\varphi_{1}, \varphi_{0}): ((\mathbb{C}^{n}\setminus \mathbb{V}(J_{\Sigma}))\times G_{min}\toto \mathbb{C}^{n}\setminus \mathbb{V}(J_{\Sigma}))\to
(Z\times G\toto Z)$$ defines a morphism between
groupoids.  Let $\varphi: [(\mathbb{C}^{n}\setminus \mathbb{V}(J_{\Sigma}))/G_{min}]\to [Z/G]$ be
the morphism of stacks induced from $(\varphi_{0}\times
\varphi_{1}, \varphi_{0})$.  

\begin{prop}[\cite{Jiang}]\Label{newstack}
The morphism $\varphi: \mathcal{X}(\mathbf{\Sigma_{min}})\to \mathcal{X}(\mathbf{\Sigma})$ is an isomorphism.  
\end{prop}

\subsection{Toric Stack Bundles}

In this section we introduce the toric stack bundle
$^{P}\mathcal{X}(\mathbf{\Sigma})$.  Let $P\to B$ be a principal
$(\mathbb{C}^{*})^{m}$-bundle over a smooth variety $B$. Let $G$ act on the
fibre product $P\times_{(\mathbb{C}^{*})^{m}}Z$ via $\alpha$ in (\ref{stack}).
\begin{defn}
Define the toric stack bundle
$^{P}\mathcal{X}(\mathbf{\Sigma})\to B$ to be the
quotient stack
\begin{equation*}
^{P}\mathcal{X}(\mathbf{\Sigma}):=[(P\times_{(\mathbb{C}^{*})^{m}}Z)/G].
\end{equation*}
\end{defn}

Let $\mathbf{\Sigma}$ be a stacky fan. For a cone $\sigma\in
\Sigma$, define 
$link(\sigma):=\{\tau:
\sigma+\tau\in \Sigma, \sigma\cap \tau=0\}.$
Let
$\{\widetilde{\rho}_{1},\ldots,\widetilde{\rho}_{l}\}$ be the rays
in $link(\sigma)$. Then
$\mathbf{\Sigma/\sigma}=(N(\sigma)=N\slash N_{\sigma},\Sigma/\sigma,\beta(\sigma))$
is a stacky fan, where $\beta(\sigma):
\mathbb{Z}^{l+m-n}\to N(\sigma)$ is given by the
images of $b_{1},\ldots,b_{l},b_{n+1},\ldots,b_{m}$ under
$N\to N(\sigma)$. From the construction of 
toric Deligne-Mumford stacks, we have
$\mathcal{X}(\mathbf{\Sigma/\sigma}):=[Z(\sigma)/G(\sigma)]$,
where
$Z(\sigma)=(\mathbb{A}^{l}\setminus \mathbb{V}(J_{\Sigma/\sigma}))\times
(\mathbb{C}^{*})^{m-n}$,
$G(\sigma)=\text{Hom}_{\mathbb{Z}}(DG(\beta(\sigma)),\mathbb{C}^{*})$.
We have an action of $(\mathbb{C}^{*})^{m}$ on
$Z(\sigma)$ induced by the natural action of
$(\mathbb{C}^{*})^{l+m-n}$ on $Z(\sigma)$ and the
projection $(\mathbb{C}^{*})^{m}\to
(\mathbb{C}^{*})^{l+m-n}$.   As in \cite{Jiang}, let
\begin{eqnarray}
^{P}\mathcal{X}(\mathbf{\Sigma/\sigma})&=&[(P\times_{(\mathbb{C}^{*})^{m}}(\mathbb{C}^{*})^{l+m-n}
\times_{(\mathbb{C}^{*})^{l+m-n}}Z(\sigma))/G(\sigma)] \nonumber \\
&=&[(P\times_{(\mathbb{C}^{*})^{m}}Z(\sigma))/G(\sigma)].
\nonumber
\end{eqnarray}

\begin{prop}[\cite{Jiang}]
Let $\sigma$ be a cone in the  stacky fan
$\mathbf{\Sigma}$,  then
$^{P}\mathcal{X}(\mathbf{\Sigma/\sigma})$ defines a closed
substack of  $^{P}\mathcal{X}(\mathbf{\Sigma})$.
\end{prop}

For each top dimensional cone $\sigma$ in
$\Sigma$, denote by $Box(\sigma)$ the set of elements $v\in
N$ such that $\overline{v}=\sum_{\rho_{i}\subseteq
\sigma}a_{i}\overline{b}_{i}$ for some $0\leq a_{i}<1$. Elements in  $Box(\sigma)$ are in one-to-one correspondence with elements in the finite group $N(\sigma)=N/N_{\sigma}$, where
$N(\sigma)$ is a local group of the stack
$\mathcal{X}(\mathbf{\Sigma})$. If $\tau\subseteq \sigma$ is a subcone, we define $Box(\tau)$ to be the set of
elements in $v\in N$ such that
$\overline{v}=\sum_{\rho_{i}\subseteq \tau}a_{i}\overline{b}_{i}$,
where $0\leq a_{i}<1$. Clearly $Box(\tau)\subset
Box(\sigma)$. In fact the elements in $Box(\tau)$ generate a
subgroup of the local group $N(\sigma)$. Let
$Box(\mathbf{\Sigma})$ be the union of $Box(\sigma)$ for all
$d$-dimensional cones $\sigma\in \Sigma$. For
$v_{1},\ldots,v_{n}\in N$, let
$\sigma(\overline{v}_{1},\ldots,\overline{v}_{n})$ be the unique
minimal cone in $\Sigma$ containing
$\overline{v}_{1},\ldots,\overline{v}_{n}$.

The following description for the inertia stack of $^{P}\mathcal{X}(\mathbf{\Sigma})$ is found in \cite{Jiang}.
\begin{prop}\Label{r-inertia}
Let $^{P}\mathcal{X}(\mathbf{\Sigma})\to B$ be a
toric stack bundle over a smooth variety $B$ with fibre
$\mathcal{X}(\mathbf{\Sigma})$, the  toric
Deligne-Mumford stack associated to the stacky fan
$\mathbf{\Sigma}$. Then its $r$-th inertia stack is
$$\mathcal{I}_{r}\left(^{P}\mathcal{X}(\mathbf{\Sigma})\right)=\coprod_{(v_{1},\cdots,v_{r})\in Box(\mathbf{\Sigma})^{r}}
~^{P}\mathcal{X}(\mathbf{\Sigma/\sigma}(\overline{v}_{1},\cdots,\overline{v}_{r})).$$
\end{prop}

\section{The Chen-Ruan Orbifold Cohomology of Toric Stack Bundles.}\Label{CR-ring}
In this section we describe the ring structure of the orbifold cohomology  of toric stack bundles.

\subsection{Orbifold Cohomology}
The Chen-Ruan Chow ring of projective toric Deligne-Mumford stacks was computed in
\cite{BCS}, and generalized to semi-projective case in \cite{JT}. The calculation for Chen-Ruan orbifold cohomology ring is the same.
In this section  we assume that the toric Deligne-Mumford stacks are semi-projective.

For $\theta\in M=N^\star$, let $\chi^{\theta}:
(\mathbb{C}^{*})^{m}\to \mathbb{C}^{*}$ be
the map induced by  $\theta\circ\beta:
\mathbb{Z}^{m}\to \mathbb{Z}$. Let $\xi_{\theta}\to B$ be the line bundle
$P\times_{\chi^{\theta}}\mathbb{C}$.  We introduce the deformed
ring $H^{*}(B)[N]^{\mathbf{\Sigma}}=H^{*}(B)\otimes
\mathbb{Q}[N]^{\mathbf{\Sigma}}$, where
$\mathbb{Q}[N]^{\mathbf{\Sigma}}:=\bigoplus_{c\in
N}\mathbb{Q}\cdot y^{c}$, $y$ is a formal variable, and $H^{*}(B)$ is the
cohomology ring of $B$. The multiplication of
$\mathbb{Q}[N]^{\mathbf{\Sigma}}$ is given by
\begin{equation}\Label{product}
y^{c_{1}}\cdot y^{c_{2}}:=\begin{cases}y^{c_{1}+c_{2}}&\text{if
there is a cone}~ \sigma\in\Sigma ~\text{such that}~ \overline{c}_{1}\in\sigma, \overline{c}_{2}\in\sigma\,,\\
0&\text{otherwise}\,.\end{cases}
\end{equation}
Let $\mathcal{I}(^{P}\mathbf{\Sigma})$ be the ideal in
$H^{*}(B)[N]^{\mathbf{\Sigma}}$ generated by the following elements:
\begin{equation}\Label{ideal-coh}
\left(c_{1}(\xi_{\theta})+\sum_{i=1}^{n}\theta(b_{i})y^{b_{i}}\right)_{\theta\in
M},
\end{equation}  and
$H_{CR}^{*}\left(^{P}\mathcal{X}(\mathbf{\Sigma})\right)$ 
the Chen-Ruan cohomology ring of the toric stack bundle $^{P}\mathcal{X}(\mathbf{\Sigma})$.

\begin{thm}[\cite{Jiang}]
Let $^{P}\mathcal{X}(\mathbf{\Sigma})\to B$ be a toric stack bundle over a smooth variety $B$ as above. We have an isomorphism of
$\mathbb{Q}$-graded rings:
$$H_{CR}^{*}\left(^{P}\mathcal{X}(\mathbf{\Sigma})\right)\cong \frac{H^{*}(B)[N]^{\mathbf{\Sigma}}}{\mathcal{I}(^{P}\mathbf{\Sigma})}.$$
\end{thm}

From the definition of Chen-Ruan cohomology ring, we have 
\begin{equation}\label{decomposition}
H_{CR}^{*}\left(^{P}\mathcal{X}(\mathbf{\Sigma})\right)=\bigoplus_{v\in Box(\mathbf{\Sigma})}H^{*}\left(^{P}\mathcal{X}(\mathbf{\Sigma/\sigma}(\overline{v}))\right) 
\end{equation}
The closed substack $^{P}\mathcal{X}(\mathbf{\Sigma/\sigma}(\overline{v}))$ is also a toric stack bundle over $B$ with fibre being the toric Deligne-Mumford stack $\mathcal{X}(\mathbf{\Sigma}/\sigma(\overline{v}))$ associated to the quotient stacky fan 
$\mathbf{\Sigma}/\sigma(\overline{v})$.  
Let 
$$link(\sigma(\overline{v}))=\{\rho_1,\cdots,\rho_l\}.$$
Let 
$I_{\mathbf{\Sigma}/\sigma(\overline{v})}$ be the ideal of $H^{*}(B)[y^{\tilde{b}_{1}},\cdots,y^{\tilde{b}_{l}}]$ generated by 
$$\{y^{\tilde{b}_{i_{1}}}\cdots y^{\tilde{b}_{i_{k}}}| \rho_{i_{1}},\cdots, \rho_{i_{k}} ~\mbox{do not span a cone in}~ \Sigma/\sigma(\overline{v})\}.$$
Then the cohomology ring of $^{P}\mathcal{X}(\mathbf{\Sigma/\sigma}(\overline{v}))$ is 
isomorphic to the Stanley-Reisner ring of the quotient fan
over the cohomology ring $H^{*}(B)$ of the base $B$:
\begin{equation}\label{ring-twisted-sector}
H^{*}(^{P}\mathcal{X}(\mathbf{\Sigma/\sigma}(\overline{v})))\cong
\frac{H^{*}(B)[y^{\tilde{b}_{1}},\cdots,y^{\tilde{b}_{l}}]}{I_{\mathbf{\Sigma}/\sigma(\overline{v})}+\mathcal{I}(^{P}\mathbf{\Sigma}/\sigma(\overline{v}))}.
\end{equation}

\begin{rmk}
As pointed out in \cite{BH}, the Chen-Ruan cohomology ring $H_{CR}^{*}\left(^{P}\mathcal{X}(\mathbf{\Sigma})\right)$ is not $Artinian$ in general if $N$ has torsion, since it has degree zero elements.  If $N$ is free, i.e. the toric Deligne-Mumford stack is reduced, then $H_{CR}^{*}\left(^{P}\mathcal{X}(\mathbf{\Sigma})\right)$  is an Artinian module 
over the cohomology ring $H^{*}(B)$ of the base.
\end{rmk}

\subsection{Obstruction Bundle}

The key gradient of Chen-Ruan orbifold cup product is the orbifold obstruction bundle 
defined over the double inertia stacks. We review it here for the latter use.

The stack $^{P}\mathcal{X}(\mathbf{\Sigma})$ is an {\em abelian} Deligne-Mumford stack,
i.e. the local groups are all abelian groups.  
The 3-twisted sector sectors of $^{P}\mathcal{X}(\mathbf{\Sigma})$ are given by 
triples $(v_1,v_2,v_3)$ for $v_1,v_2,v_3\in Box(\mathbf{\Sigma})$ such that 
$v_1+v_2+v_3$ belongs to $N$.

For any
3-twisted sector
$^{P}\mathcal{X}(\mathbf{\Sigma}/(v_1,v_2,v_3))$,
the normal bundle
$N(^{P}\mathcal{X}(\mathbf{\Sigma}/(v_1,v_2,v_3))/^{P}\mathcal{X}(\mathbf{\Sigma}))$
splits into the direct sum of line bundles under the
group action. It follows from the definition that if
$\overline{v}=\sum_{\rho_{i}\subseteq\sigma(v_1,v_2,v_3)}\alpha_{i}\overline{b}_{i}$,
then the action of $v$ on the normal bundle
$N(^{P}\mathcal{X}(\mathbf{\Sigma}/(v_1,v_2,v_3))/^{P}\mathcal{X}(\mathbf{\Sigma}))$
is given by the diagonal matrix $diag(\alpha_{i})$. 
Let $e: ~^{P}\mathcal{X}(\mathbf{\Sigma}/(v_1,v_2,v_3))\to ~^{P}\mathcal{X}(\mathbf{\Sigma})$
be the embedding. According to \cite{CR1} the obstruction bundle $Ob_{(v_1,v_2,v_3)}$ over $\mathcal{X}(\mathbf{\Sigma}/(v_1,v_2,v_3))$ is defined as
$$Ob_{(v_1,v_2,v_3)}:=(H^{1}(\mathcal{C}, \mathcal{O}_{\mathcal{C}})\otimes e^{*}T_{^{P}\mathcal{X}(\mathbf{\Sigma})})^{\langle v_1,v_2,v_3\rangle},$$
where $\langle v_1,v_2,v_3\rangle$ is the subgroup generated by 
$v_1,v_2,v_3$ and $\mathcal{C}$ is the $\langle v_1,v_2,v_3\rangle$-cover
over the Riemann sphere $\mathbb{P}^1$. Details can be found in \cite{CR1}. 
Let $v_1+v_2+v_3=\sum_{\rho_i\subset \sigma(\overline{v}_1,\overline{v}_2,\overline{v}_3)}a_ib_i$.
We will use the following description of the Euler class of the obstruction bundle:

\begin{prop}[see \cite{CH}, \cite{JKK}]\label{obstructionbdle}
Let $^{P}\mathcal{X}(\mathbf{\Sigma}/(v_1,v_2,v_3))$ 
be a 3-twisted sector of the stack  $^{P}\mathcal{X}(\mathbf{\Sigma})$  such that 
$v_{1},v_{2},v_{3}\neq 0$.
Then the Euler class of the obstruction bundle
$Ob_{(v_{1},v_{2},v_{3})}$ is
\begin{equation}\label{obstruction-bundle}
Ob_{(v_{1},v_{2},v_{3})}=\prod_{a_{i}=2}c_{1}(\mathcal{L}_{i})|_{\mathcal{X}(\mathbf{\Sigma}/(\overline{v}_{1},\overline{v}_{2},\overline{v}_{3}))},
\end{equation}
where $\mathcal{L}_{i}$ is the line bundle over $^{P}\mathcal{X}(\mathbf{\Sigma})$  determined by the ray  $\rho_{i}$.  
\end{prop}

\section{The $K$-Theory of Toric Stack Bundles}\Label{k-theory}
In this section we study the Grothendieck ring of toric stack bundles and prove the 
main theorem.

\subsection{The $K$-Theory of Toric Deligne-Mumford Stacks}
We recall the result of \cite{BH}. Let $\mathbf{\Sigma}$ be a stacky fan and $\mathcal{X}(\mathbf{\Sigma})$ the corresponding
toric Deligne-Mumford stack. 
For each ray  $\rho_{i}$ in the fan $\Sigma$, define the  line bundle $L_{i}$ over
$\mathcal{X}(\mathbf{\Sigma})$ to be the quotient of the
trivial line bundle $Z\times \mathbb{C}$ over $Z$ under
the action of $G$ on $\mathbb{C}$ through $i$-th component of
$\alpha$ in (\ref{stack}).  Let $x_i$ represent the class $[L_i]$ in the 
Grothendieck $K$-theory ring.

Let $R$ be the character ring of the group $G_{min}$.
Let $Cir(\mathbf{\Sigma})$ be the ideal in $K(B)\otimes R$
generated by the elements
\begin{equation}\label{ideal-toric}
\left(\prod_{1\leq j\leq n}x_{j}^{\langle\theta,v_{j}\rangle}-1\right)_{\theta\in
M}.
\end{equation} 
Let $I_{\mathbf{\Sigma}}$ be the ideal generated by
\begin{equation}\label{ideal2-toric}
\prod_{i\in I}(1-x_i)=0,
\end{equation}
where $I\subseteq[1,\cdots,n]$ such that $\{\rho_i| i\in I\}$ do not form a cone 
in $\Sigma$.
According to \cite{BH}, the Grothendieck $K$-theory ring $K_0(\mathcal{X}(\mathbf{\Sigma}))$ of $\mathcal{X}(\mathbf{\Sigma})$  can be described as follows.

\begin{thm}[\cite{BH}]\label{toric-stack}
For a toric Deligne-Mumford stack  $\mathcal{X}(\mathbf{\Sigma})$, the morphism
$$\phi:  \frac{R}{I_{\mathbf{\Sigma}}+Cir(\mathbf{\Sigma})}\longrightarrow K_{0}(\mathcal{X}(\mathbf{\Sigma})),$$
which send $\chi$ to $[L_{\chi}]$, is an isomorphism.
\end{thm}

Let $\mathbf{\Sigma_{min}}$ be the minimal stacky fan associated to $\mathbf{\Sigma}$. 
There is  an underlying {\em reduced} stacky
fan $\mathbf{\Sigma_{red}}=(\overline{N},\Sigma,\overline{\beta})$,
where $\overline{N}=N/N_{tor}$, $\overline{\beta}:
\mathbb{Z}^{n}\to \overline{N}$ is the natural projection given by the vectors
$\{\overline{b}_{1},\cdots,\overline{b}_n\}\subseteq \overline{N}$.
 Consider the following diagram
$$
\xymatrix{
~\mathbb{Z}^{n}\rto^{\beta}\dto_{id}&N\dto{}\\
~\mathbb{Z}^{n}\rto^{\overline{\beta}}&~\overline{N}.}
$$
Taking Gale duals yields
\begin{equation}\label{diagram-section4}
\begin{CD}
0 @ >>>\overline{N}^{\star}@ >>> \mathbb{Z}^{n}@ >{\overline{\beta}^{\vee}}>>
DG(\overline{\beta})@
>>>0@>>> 0\\
&& @VV{}V@VV{}V@VV{\varphi}V@V{}VV \\
0@ >>> N^{\star} @ >{}>>\mathbb{Z}^{n}@ >{\beta^{\vee}}>> DG(\beta)@>>>coker(\beta^{\vee})@>>>0.
\end{CD}
\end{equation}
Applying $\text{Hom}_{\mathbb{Z}}(-,\mathbb{C}^{*})$ to (\ref{diagram-section4}) yields
\begin{equation}\label{diagram-section4-2}
\begin{CD}
1 @ >>>\mu@ >>> G@ >{\alpha}>>
(\mathbb{C}^{*})^{n} @
>>>T@>>> 1\\
&& @VV{}V@VV{\alpha(\varphi)}V@VV{}V@V{}VV \\
1@ >>> 1 @ >{}>>\overline{G}@ >{\overline{\alpha}}>> (\mathbb{C}^{*})^{n}
@>>>T@>>> 1,
\end{CD}
\end{equation}

The stack $\mathcal{X}(\mathbf{\Sigma_{red}})$
is a toric orbifold. By construction 
$\mathcal{X}(\mathbf{\Sigma_{red}})=[Z/\overline{G}]$, where
$\overline{G}=\text{Hom}_{\mathbb{Z}}(DG(\overline{\beta}),\mathbb{C}^{*})$
and $DG(\overline{\beta})$ is the Gale dual 
$\overline{\beta}^{\vee}: \mathbb{Z}^{n}\to \overline{N}^{\vee}$
of the map $\overline{\beta}$. 
We can see from (\ref{diagram-section4-2}) that every character of $\overline{G}$ 
can be represented as a character of $(\mathbb{C}^{*})^{n}$.  So we have:
 
\begin{thm}\label{toric-stack-red}
For the reduced toric Deligne-Mumford stack  $\mathcal{X}(\mathbf{\Sigma_{red}})$ the morphism
$$\phi:  \frac{\mathbb{Z}[x_1,x_1^{-1},\cdots,x_n,x_n^{-1}]}{I_{\mathbf{\Sigma}}+Cir(\mathbf{\Sigma})}\longrightarrow K_{0}(\mathcal{X}(\mathbf{\Sigma_{red}})),$$
which send $x_{i}$ to $[L_{i}]$, is an isomorphism.
\end{thm}


\subsection{Proof of Theorem \ref{main}}
Let $\mathbf{\Sigma}$ be a stacky fan, and $\mathcal{X}(\mathbf{\Sigma})$ the associated 
toric Deligne-Mumford stack. Let $P\to B$ be a principle
$(\mathbb{C}^{*})^{m}$-bundle over the smooth variety $B$. Then we have the 
toric stack bundle $\pi: ~^{P}\mathcal{X}(\mathbf{\Sigma})\to B$.
For each ray $\rho_i$ in the fan $\Sigma$, we have a line bundle $L_i$ over $\mathcal{X}(\mathbf{\Sigma})$.
Twist it by the principal
$(\mathbb{C}^{*})^{m}$-bundle $P$, we get the line bundle
$\mathcal{L}_{i}$ over the toric stack bundle
$^{P}\mathcal{X}(\mathbf{\Sigma})$.

As in \cite{BCS} and \cite{Jiang}  we have a codimension one  closed substack  $\mathcal{X}(\mathbf{\Sigma}/\rho_j)\subset \mathcal{X}(\mathbf{\Sigma})$. There is a canonical section $s_j$ of the line bundle $L_{j}$ whose zero locus is $\mathcal{X}(\mathbf{\Sigma}/\rho_j)$.

Suppose that $\rho_{j_{1}},\cdots,\rho_{j_{r}}$ do not span a cone in $\Sigma$.
The section $s=(s_{j_{1}},\cdots,s_{j_{r}})$ of 
$L_{j_{1}}\oplus\cdots\oplus L_{j_{r}}$ is nowhere vanishing and extends to a nowhere vanishing section
$$P(s): ~^{P}\mathcal{X}(\mathbf{\Sigma})\longrightarrow \mathcal{L}_{j_{1}}\oplus\cdots\oplus \mathcal{L}_{j_{r}}$$
after twisting by the principle $(\mathbb{C}^{*})^{m}$-bundle $P$.
Hence by Remark 4.4 in \cite{SU},
\begin{equation}\label{relation1}
\prod_{1\leq p\leq r}(1-\mathcal{L}_{j_{p}})=0.
\end{equation}

For any $\theta\in M$, the $P$-equivariant isomorphism of bundles over $\mathcal{X}(\mathbf{\Sigma})$
$$\prod_{1\leq j\leq n}L_{j}^{\langle\theta,b_{j}\rangle}\cong L_{\theta}$$
yields an isomorphism of bundles over $^{P}\mathcal{X}(\mathbf{\Sigma})$,
$$\prod_{1\leq j\leq n}\mathcal{L}_{j}^{\langle\theta,b_{j}\rangle}\cong  \mathcal{L}_{\theta}.$$
Since $\mathcal{L}_{\theta}=\prod_{1\leq i\leq d}\xi_{i}^{-\langle\theta,v_{i}\rangle}$,
we obtain
\begin{equation}\label{relation2}
\prod_{1\leq j\leq n}\mathcal{L}_{j}^{\langle\theta,b_{j}\rangle}\cong \xi_{\theta}^{\vee}, \quad \text{where } \xi_{\theta}=\prod_{1\leq i\leq d}\xi_{i}^{\langle\theta,v_{i}\rangle}.
\end{equation}

Consider the following map 
$$\varphi:  \frac{K(B)\otimes R}{I_{\mathbf{\Sigma}}+C(^{P}\mathbf{\Sigma})}\longrightarrow K_{0}(^{P}\mathcal{X}(\mathbf{\Sigma})),\quad b\otimes \chi\mapsto [\pi^{*}b\otimes \mathcal{L}_{\chi}],\quad b\in K(B), \chi\in R.$$
We prove that $\varphi$ is surjective by induction on the dimension of $B$. 
It is obvious when $B$ is a point.  Let $U\subset B$ be a Zariski open subset and $Z=B\setminus U$. Consider the following diagram with exact rows (see \cite{toen}, Section 3.1 for the exactness of the bottom row):
\begin{equation}\label{diagram-ktheory}
\begin{CD}
K_0(Z)\otimes K(\mathcal{X}(\mathbf{\Sigma}))@ >>> K_0(B)\otimes K(\mathcal{X}(\mathbf{\Sigma}))@ >{}>>
K_0(U)\otimes K(\mathcal{X}(\mathbf{\Sigma})) \\
@VV{}V@VV{}V@VV{}V \\
K_0(\pi^{-1}Z)@ >{}>>K_0(^{P}\mathcal{X}(\mathbf{\Sigma}))@ >{}>> K_0(\pi^{-1}U),
\end{CD}
\end{equation}
where $\pi:  ~^{P}\mathcal{X}(\mathbf{\Sigma})\to B$
is the structure map.  By Lemma \ref{new-proof} below, the vertical map on the right of (\ref{diagram-ktheory}) is surjective.  
Then by induction the map
$\varphi$ is surjective since $dim(Z)< dim(B)$.

Now we prove that $\varphi$ is injective.  Let $\sum_{i=1}^{m}b_i[F_i]\in K(B)\otimes R$
such that 
$$\varphi\left(\sum_{i=1}^{m}b_i[F_i]\right)=\sum_{i=1}^{m}\pi^{*}b_i\otimes [\mathcal{F}_{i}]=0,$$  
where $\mathcal{F}_i$ is the twist of $F_i$ by the $(\mathbb{C}^{*})^{m}$-bundle $P$.
The sheaf  $\mathcal{F}_{i}$
is generated by $\mathcal{L}_{j}$'s corresponding to rays and the torsion line bundles 
corresponding to torsion subgroup in $G_{min}$.   From the relations in 
(\ref{relation1}) and (\ref{relation2}),  it is easy to see that if one of 
$b_i\neq 0$, then $\sum_{i=1}^{m}\pi^{*}b_i\otimes [\mathcal{F}_{i}]\neq 0$. 
So $\varphi$ is injective, hence is an isomorphism. The concludes the proof of Theorem \ref{main}.

\begin{lem}\label{new-proof}
Let $U$ be a smooth scheme. 
Let $[M/G]$ be a quotient stack, where $M$ admits a cellular decomposition (in the sense of \cite{SGA6}) which is $G$-equivariant. Then the map 
$$K_0(U)\otimes K_G(M)\longrightarrow K_G(U\times M)$$
is surjective.
\end{lem}
\begin{proof}
This is an $G$-equivariant version of \cite{SGA6}, Expose 0, Proposition 2.13. This may be proven by adopting the arguments in \cite{SGA6}, together with the following claims.
\end{proof}
{\bf Claim 1.}
Let $X$ be a smooth scheme with trivial $G$-action, and $G$ acts on $\mathbb{A}^1$. Let $p: X\times \mathbb{A}^1\to X$ be the projection. Then the pull-back $p^*: K_G(X)\to K_G(X\times \mathbb{A}^1)$ is surjective.
\begin{proof}[Proof of Claim 1]
Let $V$ be a $G$-equivariant vector bundle over $X\times \mathbb{A}^1$. Then by the non-equivariant version of Claim 1 (see \cite{SGA6}, Expose 0, Proposition 2.9), there is a vector bundle $V'$ over $X$ such that $V=p^*(V')$. Since $G$ acts trivially on $X$, it is easy to see that the $G$-action on $V$ naturally yields a $G$-action on $V'$, making $p^*$ $G$-equivariant.
\end{proof}
{\bf Claim 2.} 
Let $X$ be a smooth $G$-scheme and $Y\subset X$ a smooth closed subscheme preserved by $G$-action. Suppose that the quotient $[X/G]$ is a noetherian Deligne-Mumford stack. Set $U:=X\setminus Y$. Then the natural sequence $$K_G(Y)\to K_G(X)\to K_G(U)\to 0$$ is exact.
\begin{proof}[Proof of Claim 2]
The exactness in the middle is a general fact, see e.g. \cite{toen}, Section 3.1. The surjectivity of the restriction map $K_G(X)\to K_G(U)$ follows from Claim 3 below (we interpret $G$-equivariant sheaves as sheaves on the quotient stacks).
\end{proof}
{\bf Claim 3.} Let $X$ and $U$ be as in Claim 2. Let $\mathcal{F}$ be a coherent sheaf on $[U/G]$. Then there exists a coherent sheaf $\mathcal{F'}$ on $[X/G]$ such that $\mathcal{F'}|_{[U/G]}=\mathcal{F}$.
\begin{proof}[Proof of Claim 3]
Define a {\em quasi-coherent} sheaf $\bar{\mathcal{F}}$ on $[X/G]$ as follows. For an open subset $V\subset [X/G]$ define $\bar{\mathcal{F}}(V):=\mathcal{F}(V\cap [U/G])$. By construction $\bar{\mathcal{F}}|_{[U/G]}=\mathcal{F}$, which is coherent. The Claim then follows from \cite{l-mb}, Corollaire 15.5. 
\end{proof}

\section{Combinatorial Chern Character}\Label{Chern-Character}

In this section we study the Chern character homomorphism from the $K$-theory
to Chen-Ruan cohomology. 
For simplicity, we assume that the toric Deligne-Mumofrd stack 
$\mathcal{X}(\mathbf{\Sigma})$ is reduced. 

In Section \ref{Module} we generalize two results in \cite{BH}, which give the module isomorphism 
of the Chern character.  In Section \ref{ring} we use the Chern character homomorphism in \cite{JKK}
to show that the Chern character is an ring isomorphism.

\subsection{The Module Chern Character}\Label{Module}
By Theorem \ref{main}, 
\begin{equation}\label{complex_K_ring}
K_0(^{P}\mathcal{X}(\mathbf{\Sigma}),\mathbb{C}):=K_0(^{P}\mathcal{X}(\mathbf{\Sigma}))\otimes_{\mathbb{Z}}\mathbb{C} \simeq \frac{K(B)\otimes R}{I_{\mathbf{\Sigma}}+C(^{P}\mathbf{\Sigma})}\otimes\mathbb{C},
\end{equation}
where $R\cong \mathbb{C}[x_1,x_1^{-1},\cdots,x_n,x_n^{-1}]$.
Let $\tilde{R}$ denote the right-hand side of (\ref{complex_K_ring}). Again let $[\xi_i]\in K(B,\mathbb{C}):=K(B)\otimes_\mathbb{Z} \mathbb{C}$ represent the class of $\xi_i$ in the $K$-theory of $B$. The following Lemma generalizes \cite{BH}, Lemma 5.1.
\begin{lem}\label{lem1}
The maximum ideals of $\tilde{R}$ as $K(B,\mathbb{C})$-algebras are in bijective correspondence with elements of $Box(\mathbf{\Sigma})$. A box element $v=\sum_{\rho_i\subset \sigma}a_i\overline{b}_i$ corresponds
to the $n$-tuple $(y_1,\cdots,y_n)\in K(B,\mathbb{C})^{n}$ such that 
$$
y_i=\begin{cases}
e^{2\pi i a_i}\sqrt[r_i]{\xi_i^{\vee}} & \mbox{if}~ \rho_i\subset \sigma,\\
1& otherwise,
\end{cases}
$$
where
$\xi_i\in K(B,\mathbb{C})$ and $r_i$ is the order of $e^{2\pi i a_i}$.
\end{lem}
\begin{proof}
The maximal ideals of $\tilde{R}$ viewed as $K(B,\mathbb{C})$-algebras correspond to points
$(y_1,\cdots,y_n)$ in $K(B,\mathbb{C})^{n}$ such that  
\begin{equation}\label{first_equation}
\prod_{1\leq j\leq n}y_j^{\langle\theta,b_{j}\rangle}-\prod_{1\leq i\leq d}(\xi^{\vee}_{i})^{\langle\theta,v_{i}\rangle}=0
\end{equation}
and 
$$\prod_{i\in I}(1-x_i)=0$$
for $\theta$ and $I$ in (\ref{ideal}) and (\ref{ideal2}).

Suppose that the $K(B,\mathbb{C})$-point $(y_1,\cdots,y_n)$ satisfies the above condition.  Since $\prod_{i\in I}(1-x_i)=0$, there is some cone $\sigma\in \Sigma$ such that $y_i=1$ for $\rho_i$ outside the cone 
$\sigma$.  Assume that $\sigma$ is generated by rays $\rho_1,\cdots,\rho_k$.

Consider the relation (\ref{first_equation}). 
Since this relation holds for any $\theta\in M$, and $y_i=1$ for $\rho_i$ outside the cone 
$\sigma$, we can take $\theta: N_{\sigma}\to  \mathbb{Z}$, where $N_{\sigma}$ is the intersection of $N$ with the rational span of $\rho_1,\cdots,\rho_k$. Then we can choose 
$\theta$ such that $\theta(v_i)=1$, and $\theta(v_j)=0$ for $j\neq i$.  The value
$y_i$ is a $r_i$-th root of $\xi_i$ for some integer $r_i$. So $y_i=e^{2\pi i a_i}\sqrt[r_i]{\xi_i^{\vee}}$.
The relation now reads $\prod_{1\leq i\leq k}e^{2\pi i a_i \langle\theta,b_i\rangle}=1$,  and then
$\sum_{i}\langle\theta,b_i\rangle a_i \in\mathbb{Z}$ for all $\theta$.  This is equivalent to 
$v=\sum_{\rho_i\subset \sigma}a_i\overline{b}_i\in N$. So the maximal ideals are in one-to-one correspondence
to the box elements $Box(\mathbf{\Sigma})$.
\end{proof} 

In the reduced case the ring $\tilde{R}$ is an Artinian module over $K(B,\mathbb{C})$.
The localization $\tilde{R}_{v}$ can be taken as a submodule of $\tilde{R}$, which is simple. 
According to \cite{ZS}, we have

\begin{equation}
\tilde{R}:=\frac{K(B)\otimes R}{I_{\mathbf{\Sigma}}+C(^{P}\mathbf{\Sigma})}\otimes\mathbb{C}
=\bigoplus_{v\in Box(\mathbf{\Sigma})}\tilde{R}_{v}.
\end{equation}

\begin{prop}\label{lem2}
Let $v\in Box(\mathbf{\Sigma})$ and $\sigma(\overline{v})$ the minimal cone in $\Sigma$ containing 
$\overline{v}$. Then the $K(B,\mathbb{C})$-algebra $\tilde{R}_{v}$ is isomorphic to the cohomology 
of the closed substack $^{P}\mathcal{X}(\mathbf{\Sigma}/\sigma(\overline{v}))$ of the toric stack bundle 
$^{P}\mathcal{X}(\mathbf{\Sigma})$.
\end{prop}
\begin{proof}
Let $\sigma(\overline{v})$ be generated by the rays $\rho_1,\cdots,\rho_k$, and let $\overline{v}=\sum_{1\leq i\leq k}a_i\overline{b}_i$
with $a_i\in(0,1)$. For the rest of rays $\rho_{k+1},\cdots,\rho_{n}$, we may assume that $\rho_{k+1},\cdots,\rho_{l}$ are contained in some cone $\sigma'$ containing $\sigma$, and $\rho_{l+1},\cdots,\rho_n$ are not.

Now localizing gives the $K(B,\mathbb{C})$-algebra $\tilde{R}_{v}$.  Then $x_i-1$ is nilpotent for $i>k$, and 
$x_i-e^{2\pi i a_i}\sqrt[r_i]{\xi_i^{\vee}}$ is nilpotent for $1\leq i\leq k$.  Similar to  Lemma 5.2 of  \cite{BH}, let 
$$
z_i=\begin{cases}
\mbox{log}(x_i), & i>k,\\
\mbox{log}(x_i e^{-2\pi i a_i}(\sqrt[r_i]{\xi_i^{\vee}})^{-1}), & 1\leq i\leq k.
\end{cases}
$$ 
Now we work over the quotient ring $\tilde{R}_{1}$ of $\tilde{R}$ by a sufficiently high power of the 
maximal ideal.  Using the same method as in \cite{BH}, we see that $z_j=0$ in $\tilde{R}_{v}$ for $j>l$. 
And the relations 
$$\prod_{i\in I}(x_i-1)=0$$
are translated to
$$\prod_{i\in I_{\Sigma/\sigma}}z_i=0,$$
where $I_{\Sigma/\sigma}$ represents the subset of $\{k+1,\cdots,l\}$ such that 
$\{\rho_i| i\in I_{\Sigma/\sigma}\}$ are not contained in any cone of $\Sigma/\sigma$. 
(Note that $\{\rho_{k+1},\cdots,\rho_l\}$ are the link set of $\sigma$). So the relations $\prod_{i\in I}(x_i-1)=0$ determine the relations 
$\prod_{i\in I_{\Sigma/\sigma}}z_i=0$ in the quotient fan $\Sigma/\sigma$.

Let $ch: K(B,\mathbb{C})\to H^{*}(B,\mathbb{C})$
be the Chern character isomorphism from the 
$K$-theory of $B$ to the cohomology.
Then $ch(\xi_i)=e^{c_1(\xi_i)}$.

Consider the linear relations
$$\prod_{1\leq j\leq n}x_{j}^{\langle\theta,b_{j}\rangle}-\prod_{1\leq i\leq d}(\xi^{\vee}_{i})^{\langle\theta,v_{i}\rangle}=0$$
for $\theta\in M$.  Replacing the relations by $z_i$ we get
\begin{equation}\label{relation1-5}
\prod_{i=1}^{k}e^{2\pi i a_i\langle\theta,b_i\rangle}\xi_i^{\langle\theta,v_i\rangle}
\prod_{i=1}^{k+l}e^{z_i\langle\theta,b_{j}\rangle}-\prod_{1\leq i\leq d}(\xi^{\vee}_{i})^{\langle\theta,v_{i}\rangle}=0.
\end{equation}

Let $N_{\sigma(v)}$ be the sublattice generated by $\sigma(v)$, and 
$N(\sigma(v))=N/N_{\sigma(v)}$.  Let $\overline{N}(\sigma(v))$ be the free part of 
$N(\sigma(v))$, and $M(\sigma(v)):=N(\sigma(v))^{\star}$.
Consider the following diagram:
\begin{equation}\label{diagram}
\xymatrix{
N\rto^{\pi}\dto_{\theta} &N(\sigma(v))\dlto^{\tilde{\theta}}\\
\text{$\mathbb{Z}$}& }
\end{equation}
where $\pi$ is the natural morphism. For any $\tilde{\theta}\in M(\sigma(v))$, there is an element
$\theta\in M$ induced from diagram (\ref{diagram}).  Since $\xi_{\theta}=\prod_{1\leq i\leq d}\xi_{i}^{\langle\theta,v_{i}\rangle}$, and 
$e^{c_{1}(\xi_{\theta})}=ch(\xi_{\theta})$,  passing to the quotient fan 
$\Sigma/\sigma(v)$ in the lattice $\overline{N}(\sigma(v))$ the equation (\ref{relation1-5})
becomes 
$$e^{\sum_{i=k+1}^{k+l}z_i\langle\theta,b_i\rangle}-e^{c_1(\xi_{\tilde{\theta}}^{\vee})}=0.$$
So these relations yield
$$\sum_{i=k+1}^{k+l}z_i\langle\theta,b_i\rangle+c_1(\xi_{\tilde{\theta}})=0$$
which are exactly the linear relations in the cohomology ring of toric stack bundles.
Since $x_1,\cdots,x_k$ can be represented as linear combinations of $z_{k+1},\cdots,z_{l}$, the 
algebra $\tilde{R}_{v}$ is isomorphic to the ring $H^{*}(B)[z_{k+1},\cdots,z_l]$ with relations
$$\prod_{i\in I_{\Sigma/\sigma}}z_i=0,\quad \text{and}\quad \sum_{i=k+1}^{k+l}z_i\langle\theta,b_i\rangle+c_1(\xi_{\tilde{\theta}})=0.$$
So compared to the result in 
(\ref{ring-twisted-sector}),
$\tilde{R}_{v}$ is isomorphic to $H^{*}(^{P}\mathcal{X}(\mathbf{\Sigma}/\sigma),\mathbb{C})$.
\end{proof}

The decomposition 
(\ref{decomposition}) then yields the following.

\begin{thm}\label{module-iso}
Assume that the toric Deligne-Mumford stack $\mathcal{X}(\mathbf{\Sigma})$ is semi-projective. 
There is a Chern character map from $K_0(^{P}\mathcal{X}(\mathbf{\Sigma}),\mathbb{C})$
to the Chen-Ruan cohomology $H^{*}_{CR}(^{P}\mathcal{X}(\mathbf{\Sigma}),\mathbb{C})$
which is a module isomorphism.
\end{thm}
\begin{proof}
By Theorem \ref{main}, Lemma \ref{lem1} and Proposition \ref{lem2},  the Chern character map
$$ch: K_0(^{P}\mathcal{X}(\mathbf{\Sigma}),\mathbb{C})\longrightarrow
H^{*}_{CR}(^{P}\mathcal{X}(\mathbf{\Sigma}),\mathbb{C})$$
defined by $\mathcal{L}\mapsto ch(\mathcal{L})$ is a module isomorphism.
\end{proof}

\subsection{Ring Homomorphism}\Label{ring}

In this section we use the stringy $K$-theory product defined in \cite{JKK} to study the 
ring homomorphism of Chern character.

Let $^{P}\mathcal{X}(\mathbf{\Sigma})=[(P\times_{(\mathbb{C}^{*})^{m}}Z)/G]$
be the toric stack bundle associated to the stacky fan $\mathbf{\Sigma}$ and the smooth 
variety $B$.  Its $K$-theory admits the following decomposition (see e.g. \cite{AS}, \cite{AR}, \cite{VV}):
\begin{align}
K(^{P}\mathcal{X}(\mathbf{\Sigma}))&=K_{G}(P\times_{(\mathbb{C}^{*})^{m}}Z)  \nonumber \\
&=(K(I_{G}(P\times_{(\mathbb{C}^{*})^{m}}Z)))^{G}  \nonumber \\
&=\sum_{g\in G}(K(P\times_{(\mathbb{C}^{*})^{m}}Z)^{g})^{G}.
\end{align}

By Proposition \ref{r-inertia}  and  \cite{Jiang},  the twisted sectors 
$^{P}\mathcal{X}(\mathbf{\Sigma}/\sigma(\overline{v}))$
of $^{P}\mathcal{X}(\mathbf{\Sigma})$
are indexed by the box elements $v\in Box(\mathbf{\Sigma})$.
For each $v$ in the box, there exists a unique $g\in G$ such that 
$$^{P}\mathcal{X}(\mathbf{\Sigma}/\sigma(\overline{v}))\cong [(P\times_{(\mathbb{C}^{*})^{m}}Z)^{g}/G].$$

Let $\mathcal{F}_{v_1}, \mathcal{F}_{v_2}\in K_0(^{P}\mathcal{X}(\mathbf{\Sigma}))$.  The stringy $K$-theory product  of \cite{JKK} is defined by 
\begin{equation}\label{new-product}
\mathcal{F}_{v_1}\star \mathcal{F}_{v_2}=(I\circ e_3)_{*}(e_1^{*}\mathcal{F}_{v_1}\otimes e_2^{*}\mathcal{F}_{v_2}\otimes \lambda_{-1}(Ob^{*}_{v_1,v_2,v_3})),
\end{equation}
where 
$$e_i:  ~^{P}\mathcal{X}(\mathbf{\Sigma}/\sigma(v_1,v_2,v_3))\longrightarrow ~^{P}\mathcal{X}(\mathbf{\Sigma}/\sigma(v_i))$$
is the evaluation map, and $I:  ~^{P}\mathcal{X}(\mathbf{\Sigma}/\sigma(v))\to ~^{P}\mathcal{X}(\mathbf{\Sigma}/\sigma(v^{-1}))$
is the involution map. Needless to say, the stringy $K$-theory product is defined in a way very similar to that of Chen-Ruan cup product.

Let $^{P}\mathcal{X}(\mathbf{\Sigma}/\sigma(v))$ be a twisted sector and  $W_{v}=T_{^{P}\mathcal{X}(\mathbf{\Sigma})}|_{^{P}\mathcal{X}(\mathbf{\Sigma}/\sigma(v))}$.   Define $W_{v,k}$ to be 
the eigenbundle of $W_{v}$,  where $v$ acts by multiplication by $\zeta^{k}=e^{2\pi i k/r}$.
Following \cite{JKK},  we define 
\begin{equation*}
\mathcal{T}_{v}:=\bigoplus_{k=0}^{r-1}\frac{k}{r}W_{v,k}.
\end{equation*}
For $v=\sum_{\rho_i\subset \sigma(\overline{v})}\alpha_i b_i$, this reads
\begin{equation*}
\mathcal{T}_{v}=\bigoplus_{\rho_i\subset\sigma(\overline{v})}\alpha_i \mathcal{L}_{i}.
\end{equation*}

\begin{thm}
The ``stringy'' Chern character morphism 
$$ch_{orb}: K_0(^{P}\mathcal{X}(\mathbf{\Sigma}),\mathbb{C})\longrightarrow H^{*}_{CR}(^{P}\mathcal{X}(\mathbf{\Sigma}),\mathbb{C}), \quad ch_{orb}(\mathcal{F}_{v}):=ch(\mathcal{F}_{v})td^{-1}\mathcal{T}_{v}$$
is an isomorphism as rings under the stringy $K$-theory product.
\end{thm}
\begin{proof}
This is a special case of \cite{JKK}, Theorem 9.5. By Theorem \ref{module-iso}, the morphism is a module isomorphism. It remains to check the product.  
Let $\mathcal{L}_i$ and $\mathcal{L}_j$ be line bundles over $^{P}\mathcal{X}(\mathbf{\Sigma})$ as defined before, then we have
\begin{equation*}
\begin{split}
ch_{orb}(\mathcal{L}_i\star \mathcal{L}_j)&=ch(\mathcal{L}_i\otimes \mathcal{L}_j\otimes \lambda_{-1}(\oplus_{a_i=2}\mathcal{L}_{i})^{*})\cdot td^{-1}\mathcal{T}_{v_{3}^{-1}},\\
ch_{orb}(\mathcal{L}_i)\cup_{CR}ch_{orb}(\mathcal{L}_j)&=ch(\mathcal{L}_i)\cdot td^{-1}\mathcal{T}_{v_{1}} ch(\mathcal{L}_j)\cdot td^{-1}\mathcal{T}_{v_{2}} \cdot e(\oplus_{a_i=2}\mathcal{L}_{i}).
\end{split}
\end{equation*}
Since 
\begin{equation*}
\begin{split}
td(\oplus_{a_i=2}\mathcal{L}_{i})\cdot ch(\lambda_{-1}(\oplus_{a_i=2}\mathcal{L}_{i})^{*})&=
e(\oplus_{a_i=2}\mathcal{L}_{i}),\\
td^{-1}\mathcal{T}_{v_{1}} \cdot td^{-1}\mathcal{T}_{v_{2}} \cdot td(\oplus_{a_i=2}\mathcal{L}_{i})&=
td^{-1}\mathcal{T}_{v_{3}^{-1}},
\end{split}
\end{equation*}
we conclude
$ch_{orb}(\mathcal{L}_i\star \mathcal{L}_j)=
ch_{orb}(\mathcal{L}_i)\cup_{CR}ch_{orb}(\mathcal{L}_j).$
\end{proof}
\section{Example: Finite Abelian Gerbes}\Label{gerbes}

In \cite{Jiang}, the degenerate case of
toric stack bundles, namely finite abelian gerbes over smooth varieties, were studied.  
In this section we compute their $K$-theory. We first recall the construction of finite abelian gerbes.

Let
$N=\mathbb{Z}_{p_{1}^{n_{1}}}\oplus\cdots\oplus\mathbb{Z}_{p_{s}^{n_{s}}}$
be a finite abelian group, where $p_{1},\cdots,p_{s}$ are prime
numbers and $n_{1},\cdots,n_{s}>1$. Let $\beta:
\mathbb{Z}\to N$ be given by the vector
$(1,1,\cdots,1)$. $N_{\mathbb{Q}}=0$ implies that $\Sigma=0$, then
$\mathbf{\Sigma}=(N,\Sigma,\beta)$ is a stacky
fan. Let
$n=lcm(p_{1}^{n_{1}},\cdots,p_{s}^{n_{s}})$, then
$n=p_{i_{1}}^{n_{i_{1}}}\cdots p_{i_{t}}^{n_{i_{t}}}$, where
$p_{i_{1}},\cdots, p_{i_{t}}$ are the distinct prime numbers which
have the highest powers $n_{i_{1}},\cdots, n_{i_{t}}$. Note that
the vector $(1,1,\cdots,1)$ generates an order $n$ cyclic subgroup
of $N$. We calculate  the Gale dual $\beta^{\vee}:
\mathbb{Z}\stackrel{}{\to} \mathbb{Z}\oplus
\bigoplus_{i\notin
\{i_{1},\cdots,i_{t}\}}\mathbb{Z}_{p_{i}}^{n_{i}}$, where
$DG(\beta)=\mathbb{Z}\oplus \bigoplus_{i\notin
\{i_{1},\cdots,i_{t}\}}\mathbb{Z}_{p_{i}}^{n_{i}}$. We have the
following exact sequence:
$$0\longrightarrow \mathbb{Z}\longrightarrow \mathbb{Z}\stackrel{\beta}{\longrightarrow}N\longrightarrow
\bigoplus_{i\notin
\{i_{1},\cdots,i_{t}\}}\mathbb{Z}_{p_{i}}^{n_{i}}\longrightarrow
0,$$
$$0\longrightarrow 0\longrightarrow \mathbb{Z}\stackrel{(\beta)^{\vee}}{\longrightarrow}\mathbb{Z}\oplus
\bigoplus_{i\notin
\{i_{1},\cdots,i_{t}\}}\mathbb{Z}_{p_{i}}^{n_{i}}\longrightarrow
\mathbb{Z}_{n}\oplus\bigoplus_{i\notin
\{i_{1},\cdots,i_{t}\}}\mathbb{Z}_{p_{i}}^{n_{i}}\longrightarrow
0.$$ 
So we obtain \begin{equation}\Label{gerbe}1\longrightarrow
\mu\longrightarrow \mathbb{C}^{*}\times \prod_{i\notin
\{i_{1},\cdots,i_{t}\}}\mu_{p_{i}}^{n_{i}}\stackrel{\alpha}{\longrightarrow}\mathbb{C}^{*}\longrightarrow
1, \end{equation} where the map $\alpha$ in (\ref{gerbe}) is
given by the matrix  $[n,0,\cdots,0]^{t}$ and $\mu=\mu_{n}\times\prod_{i\notin
\{i_{1},\cdots,i_{t}\}}\mu_{p_{i}}^{n_{i}}\cong N$. 
The toric Deligne-Mumford stack associated with the data is
$$\mathcal{X}(\mathbf{\Sigma})=[\mathbb{C}^{*}/\mathbb{C}^{*}\times
\prod_{i\notin
\{i_{1},\cdots,i_{t}\}}\mu_{p_{i}}^{n_{i}}]=\mathcal{B}\mu,$$
i.e.
the classifying stack of the group $\mu$. 

Let $L$ be a line bundle
over a smooth variety $B$ and $L^{*}$ the principal
$\mathbb{C}^{*}$-bundle induced from $L$ removing the zero
section.  From our twist we have
$$^{L^{*}}\mathcal{X}(\mathbf{\Sigma})=L^{*}\times_{\mathbb{C}^{*}}[\mathbb{C}^{*}/\mathbb{C}^{*}\times
\prod_{i\notin \{i_{1},\cdots,i_{t}\}}\mu_{p_{i}}^{n_{i}}]
=[L^{*}/\mathbb{C}^{*}\times \prod_{i\notin
\{i_{1},\cdots,i_{t}\}}\mu_{p_{i}}^{n_{i}}],$$
 which is a $\mu$-gerbe $\mathcal{X}$ over $B$. 

\begin{rmk}
The structure of this gerbe is
a $\mu_{n}$-gerbe coming from the line bundle $L$ plus a trivial
$\prod_{i\notin \{i_{1},\cdots,i_{t}\}}\mu_{p_{i}}^{n_{i}}$-gerbe
over $B$. 
\end{rmk}

For this toric stack bundle,
$Box(\mathbf{\Sigma})=N$,  the Chen-Ruan cohomology  was computed in
\cite{Jiang}.

\begin{prop}[\cite{Jiang}]\label{orbifold-coh}
The Chen-Ruan cohomology ring of the finite abelian $\mu$-gerbe
$\mathcal{X}$ is:
$$H^{*}_{CR}(\mathcal{X},\mathbb{Q})\cong H^{*}(B,\mathbb{Q})\otimes H^{*}_{CR}(\mathcal{B}\mu,\mathbb{Q}),$$
where
$H^{*}_{CR}(\mathcal{B}\mu;\mathbb{Q})=\mathbb{Q}[t_{1},\cdots,t_{s}]/(t_{1}^{p_{1}^{n_{1}}}-1,\cdots,t_{s}^{p_{s}^{n_{s}}}-1)$.
\end{prop}

For the stacky fan $\mathbf{\Sigma}=(N,0,\beta)$, the minimal stacky fan is given by
$\mathbf{\Sigma_{min}}=(N,0,\beta_{min})$, where $\beta_{min}=0: 0\to N$
is the zero map. So the Gale dual map is still the map
$\beta_{min}$, and
$$G_{min}=\text{Hom}_{\mathbb{Z}}(N,\mathbb{C}^{*})\cong \mu.$$

The characters of $\mu\simeq N$ are  given by all the maps $\chi: \mu\to \mathbb{C}^{*}$.
Since $N=\mathbb{Z}_{p_{1}^{n_{1}}}\oplus\cdots\oplus\mathbb{Z}_{p_{s}^{n_{s}}}$, let 
$\chi_{1},\cdots,\chi_{s}$ be the base generators of the  characters of $N$ such that 
$\chi_{1}^{p_{1}^{n_{1}}},\cdots,\chi_{1}^{p_{s}^{n_{s}}}$ are trivial. 
Every character $\chi_i$ determines a line bundle $\mathcal{L}_{i}$ over $\mathcal{X}$
such that $\mathcal{L}_{1}^{p_{i}^{n_{i}}}$ is trivial. 
Then Theorem \ref{main} implies

\begin{thm}\label{k-gerbe}
The $K$-theory ring of the finite abelian gerbe $\mathcal{X}$ is:
$$K_0(\mathcal{X})\simeq\frac{K(B)[\mathcal{L}_{1},\cdots,\mathcal{L}_{s}]}
{(\mathcal{L}_{1}^{p_{1}^{n_{1}}},\cdots,\mathcal{L}_{s}^{p_{s}^{n_{s}}})}.$$
\end{thm}

\begin{rmk}
It is easy to see from Theorem \ref{k-gerbe} the $K$-theory ring of the finite abelian gerbes
is independent to the triviality and nontriviality  of the gerbes.
\end{rmk}

By Theorem \ref{orbifold-coh} and \ref{k-gerbe} we have:

\begin{thm}
There exists a Chern character morphism from the $K$-theory ring $K_0(\mathcal{X},\mathbb{C})$ of the finite abelian 
$\mu$-gerbe $\mathcal{X}$ to the Chen-Ruan cohomology $H_{CR}^{*}(\mathcal{X},\mathbb{C})$, which is 
a ring isomorphism.  
\end{thm}

\begin{rmk}
Suppose that we have two finite abelian $\mu$-gerbes over $B$, one is trivial and 
the other is nontrivial. 
We see that the $K$-theory ring and the Chen-Ruan cohomology ring cannot distinguish these two 
different stacks. However {\em quantum} cohomology rings of different gerbes are different in general  \cite{AJT1}.
\end{rmk}


\subsection*{}

\end{document}